\documentclass[12pt]{article}
\usepackage{amsmath,amssymb}
\usepackage{tikz}
\usepackage{color}
\usepackage[toc]{appendix}
\usepackage{graphicx}
\usepackage{fancyhdr}
\usepackage{enumitem}
\usepackage{bbm}
\usepackage{graphicx}
\usepackage{parskip}
\usepackage{float}
\usepackage{chngpage}
\usepackage{calc}
\usepackage{bigints}
\usepackage{tikz}
\usepackage{array}
\usepackage{booktabs}
\usepackage{rotating}
\usepackage{multirow}
\usepackage{adjustbox}
\usepackage{tabularx}
\usepackage{verbatim}
\usepackage{mathtools}
\usepackage{ragged2e}
\usepackage[makeroom]{cancel}
\usepackage{enumitem}

\newtheorem{theorem}{Theorem}[section]
\newtheorem{proposition}[theorem]{Proposition}

\newtheorem{example}[theorem]{Example}
\newtheorem{remark}[theorem]{Remark}

\newenvironment{proof}{\smallskip\par{\sc Proof.}\enspace}%
 {{\unskip\nobreak\hfil\penalty50\hskip2em
          \hbox{}\nobreak\hfil{\rule[-1pt]{5pt}{10pt}}
          \parfillskip=0pt\finalhyphendemerits=0
          \par\medskip}} 

\makeatletter
\def\section{\@startsection {section}{1}{\z@}{3.25ex plus 1ex minus
 .2ex}{1.5ex plus .2ex}{\large\bf}}
\def\subsection{\@startsection{subsection}{2}{\z@}{3.25ex plus 1ex minus
 .2ex}{1.5ex plus .2ex}{\normalsize\bf}}
\@addtoreset{equation}{section} 
\makeatother

\chardef\bslash=`\\ 

\def\R{\mathbb{R}}

\newcommand{\eval}[2][\right]{\relax
\ifx#1\right\relax \left.\fi#2#1\rvert}
 
\numberwithin{equation}{section}
\newcount\madechanges
\def\caution#1{\ifnum \madechanges=1 \affixmessage{#1}%
\else \relax \fi}
\def\affixmessage#1{\marginpar{{\footnotesize  \em #1} \openup
    -.3\baselineskip }}
\title{Stochastic perturbation of a cubic anharmonic oscillator}

\author{Enrico Bernardi\thanks{Dipartimento di Scienze Statistiche Paolo Fortunati, Università di Bologna,
Bologna, Italy. \textbf{e-mail}: enrico.bernardi@unibo.it} \and  Alberto Lanconelli\thanks{Dipartimento di Scienze Statistiche Paolo Fortunati, Università di Bologna,
Bologna, Italy. \textbf{e-mail}: alberto.lanconelli2@unibo.it}}

\date{\today}

\begin{document}

\maketitle

\numberwithin{equation}{section}

\bigskip

\begin{abstract}
We perturb with an additive Gaussian white noise the Hamiltonian system associated to a cubic anharmonic oscillator. The stochastic system is assumed to start from initial conditions that guarantee the existence of a periodic solution for the unperturbed equation. We write a formal expansion in powers of the diffusion parameter for the candidate solution and analyze the probabilistic properties of the sequence of the coefficients. It turns out that such coefficients are the unique strong solutions of stochastic perturbations of the famous Lam\'e's equation. We obtain explicit solutions in terms of Jacobi elliptic functions and prove a lower bound for the probability that an approximated version of the solution of the stochastic system stay close to the solution of the deterministic problem. Conditions for the convergence of the expansion are also provided.    
\end{abstract}



\allowdisplaybreaks

\section{Introduction}\label{intro}

We investigate the second order stochastic differential equation
\begin{eqnarray}\label{SDE}
\ddot{x}(t)=x(t)^2-\mathcal{B}+\sigma\dot{B}(t),\quad x(0)=y\quad \dot{x}(0)=\eta
\end{eqnarray} 
where $\sigma$ is a positive constant, $\mathcal{B} \in \R$ and $\{B(t)\}_{t\geq 0}$ is a one dimensional standard Brownian motion defined on the probability space $(\Omega,\mathcal{F},\mathbb{P})$ which is assumed to fulfill the usual completeness requirement. The rigorous formulation of equation (\ref{SDE}) is achieved by considering the  It\^ o-type stochastic Hamiltonian system
\begin{equation}\label{eqn:sys}
\begin{cases}
dx(t) = \xi(t) dt, & x(0)=y\\
d\xi(t)= \left(x^2(t) - \mathcal{B}\right)dt + \sigma dB(t), & \xi(0)=\eta.
\end{cases}
\end{equation}
This very simple model comes essentially from  the Hamiltonian
\begin{eqnarray}\label{hamiltonian}
H(x,\xi) = \xi^{2}/2 - x^{3}/3+\mathcal{B}x,
\end{eqnarray}
of a \emph{cubic anharmonic oscillator} perturbed by a Brownian noise. We may assume the initial data $(y,\eta) $ to be deterministic. 


The function (\ref{hamiltonian}) belongs to a class of Hamiltonians very thoroughly
studied, see e.g. Delabaere and Trinh \cite{DT},  Ferreira and Sesma \cite{FS} and the references therein, providing examples of quantum systems of interest in quantum field theory and in general theoretical
physics, especially in the $ \mathcal{PT}$-symmetric cases when
$\mathcal{B}  $ may be complex. Perhaps slightly less known is the fact that (\ref{hamiltonian}) also arises
as the major ingredient in a special canonical form in the theory of hyperbolic operators with
double characteristics of non-effectively hyperbolic type according to
H\"ormander's classification \cite{HH}, generating  at the same time a number of
problems on the regularity of the corresponding solutions and on the
behavior of the solutions of the associated Hamilton equations. We refer to Nishitani \cite{N1}, Bernardi and Nishitani \cite{BN1}, \cite{BN2} for a complete analysis of this issue: here a symplectically
invariant condition equivalent to the presence in (\ref{hamiltonian})
of the term $ -x^{3}/3 $ drives both the Gevrey classes to which the
solutions of the associated PDEs belong and determines the geometric behavior
of the simple bicharacteristics close to the double manifold. Since
the deterministic picture for the related dynamical system appears on
the whole to
be generally well understood, it seems only natural to start investigating
the stochastic system (\ref{eqn:sys}) as one of the simplest non
trivial cases  where the interaction between the white noise and the
unstable periodic solutions of the deterministic problem when $ \sigma
= 0$ in (\ref{eqn:sys}) can be seen directly.

The system (\ref{eqn:sys}) is characterized by a drift vector which is linear in the first component and quadratic in the second one and a degenerate diffusion matrix as the first equation is not perturbed by a Brownian term. The local lipschitzianity of the drift entails existence of strong solutions up to a possible almost surely finite stopping time at which the solution explodes. In addition, there are at least two distinguishing features of the system (\ref{eqn:sys}) that prevent the use of standard techniques in the analysis of existence and uniqueness of weak/strong solutions. First of all, due to the superlinear growth of the drift coefficient we are not allowed to employ the Girsanov theorem to construct weak solutions; this is one of the key tools for the investigation of almost sure properties of the solution (see Markus and Weerasinghe \cite{Markus 1}, \cite{Markus 2}). Secondly, the Hamiltonian (\ref{hamiltonian}) is not lower bounded and therefore classical methods based on the positivity of the energy have to be excluded (see Albeverio et al. \cite{Albeverio}, \cite{Albeverio 2}). \\

Our approach is based on a formal expansion in powers of $\sigma$ for
the candidate solution $\{x(t)\}_{t\geq 0}$ which permits a
qualitative study of each single coefficient of the power series. For
a general survey of how to establish the validity of the asymptotic expansion in powers
of $ \sigma $ see e.g. Gardiner \cite{Gard} page 182. To be more specific, let
\begin{eqnarray}\label{series}
x(t):=\sum_{n\geq 0}\sigma^n x_n(t),\quad t\geq 0.
\end{eqnarray}
A formal substitution of this expression into equation (\ref{SDE}) results in an equality between two power series. If we impose the coefficients of the corresponding powers of $\sigma$ to be equal, we end up with a sequence of nested random/stochastic Cauchy problems for the sequence of functions $\{x_n(t)\}_{t\geq 0, n\geq 0}$ . In fact, via a direct verification one gets that $\{x_0(t)\}_{t\geq 0}$ is associated with the deterministic equation
\begin{eqnarray}\label{x_0}
\ddot{x}_0(t)=x_0(t)^2-\mathcal{B},\quad x_0(0)=y\quad \dot{x}_0(0)=\eta.
\end{eqnarray} 
The function $\{x_1(t)\}_{t\geq 0}$ is linked to the linear stochastic differential equation
\begin{eqnarray}\label{x_1}
\ddot{x}_1(t)=2x_0(t)x_1(t)+\dot{B}(t),\quad x_1(0)=0\quad \dot{x}_1(0)=0
\end{eqnarray} 
where $\{x_0(t)\}_{t\geq 0}$ solves (\ref{x_0}). Equation (\ref{x_1}) can
be interpreted as a \emph{stochastic Lam\'e's equation}, see e.g. Arscott \cite{A}, Arscott and Khabaza
\cite{AK}, Volker \cite{Volker} and the website \cite{SL} for the deterministic case. We will see later in fact that, once $ x_{0}(t) $ is solved in (\ref{x_0}), (\ref{x_1}) can be written as 
\begin{equation}
\label{lm1}
\ddot{x}_1(t) + (h - \nu(\nu+1)k^{2}\text{sn}^{2}(t,k))x_1(t) =\dot{B}(t)~,
\end{equation}
where $ \text{sn}(t,k) $ is the Jacobi elliptic function with modulus
$ k $ and $ h,\nu $ are suitable constants depending on $ \mathcal{B}
$ essentially.\\
For $n\geq 2$ the function $\{x_n(t)\}_{t\geq 0}$ solves the random differential equation
\begin{eqnarray}\label{x_n}
\ddot{x}_n(t)=2x_0(t)x_n(t)+\sum_{j=1}^{n-1}x_j(t)x_{n-j}(t),\quad x_n(0)=0\quad \dot{x}_n(0)=0.
\end{eqnarray} 
We note that also in this case the equation to be solved is linear,  the function $\{x_0(t)\}_{t\geq 0}$ solves (\ref{x_0}) and the functions involved in the sum are the coefficients of lower order terms (with respect to the unknown) from the expansion (\ref{series}).   

\begin{remark}
The techniques employed in this paper carry over the case of a multiplicative noise term as well, that means we are able to treat with only minor and straightforward modifications also the second order stochastic differential equation
\begin{eqnarray}\label{SDE general}
\ddot{x}(t)=x(t)^2-\mathcal{B}+\sigma x(t)\dot{B}(t),\quad x(0)=y\quad \dot{x}(0)=\eta
\end{eqnarray} 
where now the white noise is multiplied by the unknown $x(t)$. In fact, proceeding as explained above with the formal substitution of the power series (\ref{series}) in the equation (\ref{SDE general}) one immediately finds that  $\{x_0(t)\}_{t\geq 0}$ is again associated with the deterministic equation (\ref{x_0}) while $\{x_1(t)\}_{t\geq 0}$ is now linked to the linear stochastic differential equation
\begin{eqnarray}\label{x_1 general}
\ddot{x}_1(t)=2x_0(t)x_1(t)+x_0(t)\dot{B}(t),\quad x_1(0)=0\quad \dot{x}_1(0)=0.
\end{eqnarray} 
We observe that the noise term in (\ref{x_1 general}) is additive and corresponds to a Brownian motion composed with a deterministic time change. Therefore, equation (\ref{x_1 general}) is a minor modification of equation (\ref{x_1}). Moreover, for $n\geq 2$ the function $\{x_n(t)\}_{t\geq 0}$ now solves the random differential equation
\begin{eqnarray}\label{x_n general}
\ddot{x}_n(t)=2x_0(t)x_n(t)+\sum_{j=1}^{n-1}x_j(t)x_{n-j}(t)+x_{n-1}(t)\dot{B}(t),\quad x_n(0)=0\quad \dot{x}_n(0)=0.
\end{eqnarray} 
The term $x_{n-1}(t)\dot{B}(t)$ in (\ref{x_n general}), which is not present in (\ref{x_n}) can be defined as a pathwise integral since the function $t\mapsto x_1(t)$ is almost surely continuously differentiable; as for the solution of (\ref{x_n general}), the new term can be absorbed by the sum on the right hand side and the analysis follows from the one employed for (\ref{x_n}).
\end{remark}

Our strategy will be to first solve explicitely (\ref{x_0}). Then (\ref{x_1}) will be seen to belong to a class of linear
stochastic Lam\'e's equations. For that we will first write down two linearly independent solutions of the homogenous equation and briefly recall the known Floquet-Lyapunouv results. Then, we will study the white-noise oscillator, proving a number of
results on the oscillatory behavior of the solution process. Finally, we will present some results on the global development $
\sum_{n\geq 0}\sigma^{n}x_n(t) $ for some special classes of noise.\\

The paper is organized as follows: In Section 2 we analyze the coefficients of the power series (\ref{series}) as solutions to certain stochastic/random differential equations; we will provide explicit solutions and describe their fundamental probabilistic properties. Then, in Section 3 we present several lower bounds in terms of the (explicit) two independent solutions of the Lam\'e's equation for both the coefficients of series (\ref{series}) and its truncated version. Finally, Section 4 provides a sufficient condition on the driving noise ensuring the almost sure uniform convergence of the series (\ref{series}) on a compact time interval whose length depend on the diffusion coefficient $\sigma$. 

\section{Analysis of the coefficients of the series (\ref{series})}

In this section we analyze the sequence of functions $\{x_n(t)\}_{n\geq 0, t\geq 0}$ appearing in the formal expansion (\ref{series}) as explained in the Introduction.  

\subsection{The equation for $x_0$: deterministic case }\label{deterministic section}

We begin with the study of the deterministic system
\begin{eqnarray*}
\ddot{x}_0(t)=x_0(t)^2-\mathcal{B},\quad x_0(0)=y\quad \dot{x}_0(0)=\eta
\end{eqnarray*} 
which is equivalent to 
\begin{equation}\label{eq:25 bis}
\begin{cases}
\dot{x}_0(t) = \xi_0(t), & x_0(0)=y\\
\dot{\xi}_0(t)= x_0^2(t) - \mathcal{B}, & \xi_0(0)=\eta.
\end{cases}
\end{equation}
The constant $\mathcal{B}$ and the initial data $(y,\eta) $ will be chosen in such a way that the third order
polynomial $x^{3}/3 - \mathcal{B} x + H(y,\eta) $ has three real roots. This implies the existence for (\ref{eq:25 bis}) of a \emph{periodic solution}, whose behavior under the stochastic perturbation is our
concern here. 

\begin{figure}[ht]
  \centering
  \includegraphics[scale=0.3]{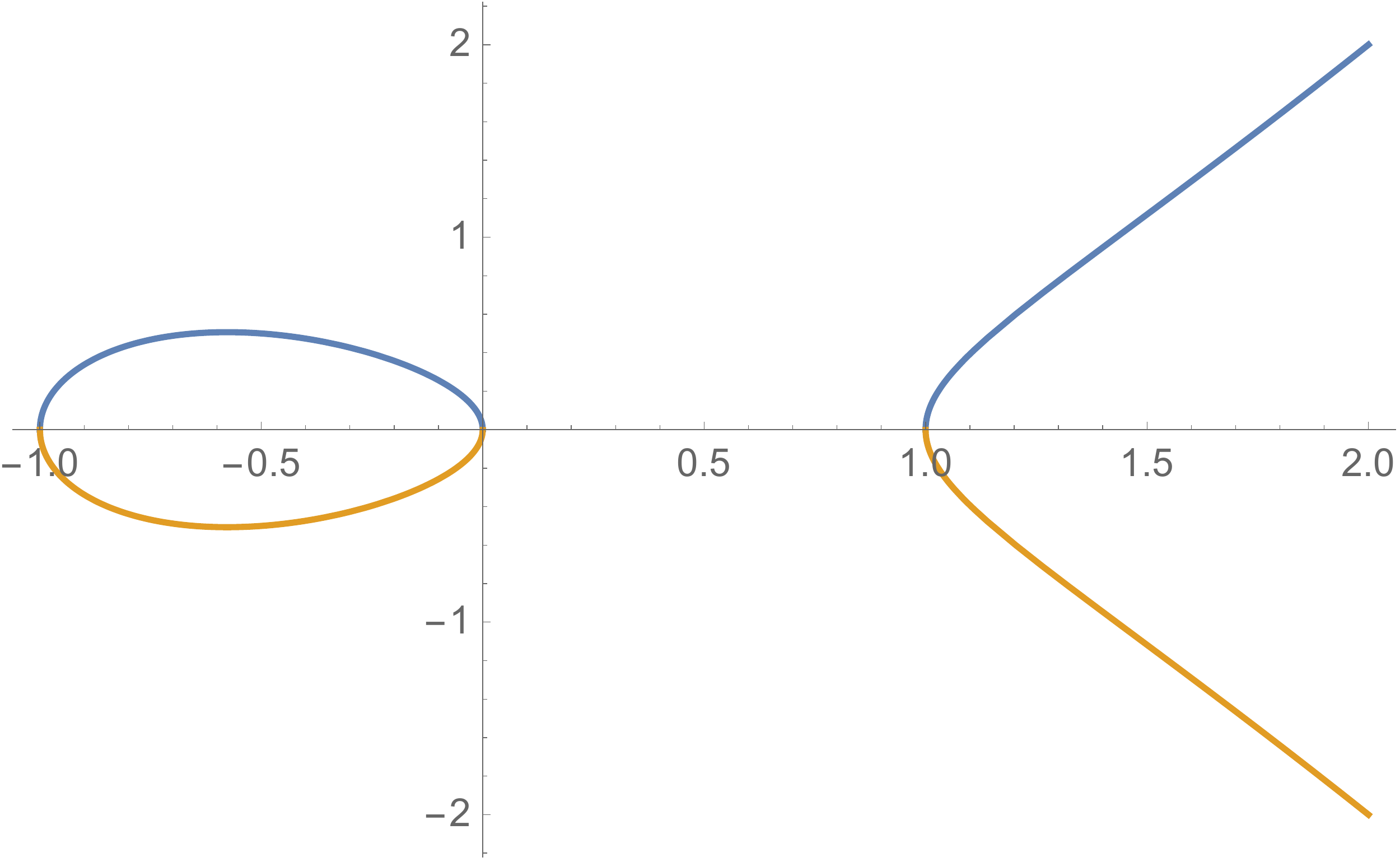}
  \caption{Graph of Hamiltonian with $ c=-1,a=1 $}
\end{figure}

Then, slightly changing our notations, we start directly with the three real roots of the polynomial and we
denote them by $c$, $-a-c$, $a $ with $ c < 0 < a $; imposing
\begin{eqnarray*}
x^{3}/3 - \mathcal{B} x + H(y,\eta) =(x-a)(x+a+c)(x-c)/3
\end{eqnarray*}
we get
\begin{eqnarray*}
\mathcal{B}=(a^2+c^2+ac)/3\quad\mbox{ and }\quad H(y,\eta)=ac(a+c)/3.
\end{eqnarray*}
Therefore, the Hamiltonian system we are going to analyze is
\begin{equation}\label{eqn:sys-det}
\begin{cases}
\dot{x}_{0}(t) = \xi_{0}(t),  &  x_{0}(0) =y\\
\dot{\xi}_{0}(t)= x_{0}^2(t) - (a^{2} + c^{2} +ac)/3, & \xi_{0}(0) = \eta,
\end{cases}
\end{equation}
with  $y\in [c, -a -c] $. We assume without loss of generality that $2c + a < 0$, which entails $c <-a -c< a$. Now, rewriting the conservation of energy
\begin{eqnarray*}
H(y,\eta)=\dot{x}_0(t)^2/2-x_0(t)^3/3+\mathcal{B}x_0(t)
\end{eqnarray*}
as
\begin{eqnarray*}
\dot{x}_0(t)^2/2=x_0(t)^3/3-\mathcal{B}x_0(t)+H(y,\eta)
\end{eqnarray*}
we get
\begin{eqnarray*}
\dot{x}_{0}(t)^{2}/2 =(x_0(t)-a)(x_0(t)+a+c)(x_0(t)-c)/3,\quad x_{0}(0) = y
\end{eqnarray*}
which in turn implies
\begin{eqnarray}\label{elliptic}
\int_{y}^{x_{0}(t)}\frac{dv}{\sqrt{(v-a)(v+a+c)(v-c)}} = \sqrt{2/3} t.
\end{eqnarray}
The integral in equation (\ref{elliptic}) is related to elliptic integrals. It is in fact known (see for instance the book by Gradshteyn and Ryzhik \cite{GR}) that
\begin{equation}\label{eq:3}
\int_{c}^{u}\frac{dv}{\sqrt{(v-a)(v-b)(v-c)}} = \frac{2}{\sqrt{a -c}}\int_{0}^{\gamma}\frac{d\alpha}{\sqrt{1 - q^{2}\sin^{2}\alpha}}
\end{equation}
whenever $c<u\leq b<a$. Here $\gamma$ and $q$ are defined by the formulas
\begin{eqnarray*}
\gamma=: \arcsin\sqrt{\frac{u-c}{b-c}}\quad\mbox{ and }\quad q : = \sqrt{\frac{b-c}{a-c}}.
\end{eqnarray*} 
In the sequel we set  
\begin{eqnarray}\label{F}
\displaystyle F(\gamma,q): =\int_{0}^{\gamma}\frac{d\alpha}{\sqrt{1 - q^{2}\sin^{2}\alpha}}
\end{eqnarray}
for the so-called elliptic integral of the first kind. We also recall that 
\begin{eqnarray}\label{sn}
\displaystyle u = F(\gamma,q) \quad\mbox{ is equivalent to }\quad\displaystyle \text{sn}(u,q) = \sin \gamma,
\end{eqnarray}
where $ \text{sn}(u,q)$ denotes the \emph{Jacobi elliptic sine}
function with modulus $ q $. While referring to Gradshteyn and Ryzhik
\cite{GR} or the website \cite{SF} for a complete exposition of the Jacobi elliptic functions, we sum up in the Appendix (\ref{appendix}) some of the essential features we will be using in the following. Therefore, comparing (\ref{elliptic}) with (\ref{eq:3}) (where we set $b=-a-c$) we get
\begin{equation*}\label{sol-1}
\sqrt{\frac{2}{3}}~t = \frac{2}{\sqrt{a-c}}F\left(\arcsin\sqrt{\frac{c -x_{0}(t)}{2c +a}},q\right) - \frac{2}{\sqrt{a-c}}F\left(\arcsin\sqrt{\frac{c -y}{2c +a}},q\right),
\end{equation*}
with $ q = \sqrt{\frac{2c+a}{c-a}} $. It is then easy to see that the last identity combined with (\ref{sn}) gives
\begin{equation}\label{eq:6-1}
x_{0}(t) = c - (a+2c)\text{sn}^{2}\left(\sqrt{\frac{a-c}{6}}t + i_{y},q\right),
\end{equation}
where we denote
\begin{eqnarray}\label{i_y}
\displaystyle i_{y}:=F\left(\arcsin\sqrt{\frac{c -y}{2c +a}},q\right).
\end{eqnarray}
We have therefore proved the following.

\begin{theorem}
The unique solution $\{x_0(t)\}_{t\geq 0}$ of the deterministic Hamiltonian system (\ref{eqn:sys-det}) where 
\begin{eqnarray*}
c<0<a,\quad 2c+a<0\quad\mbox{and}\quad y\in [c,-a-c]
\end{eqnarray*}
is explicitly given by formula (\ref{eq:6-1}) with $ q = \sqrt{\frac{2c+a}{c-a}} $ and $i_y$ defined by (\ref{i_y}).  
\end{theorem}

\begin{figure}[ht]
  \centering
  \includegraphics[scale=0.3]{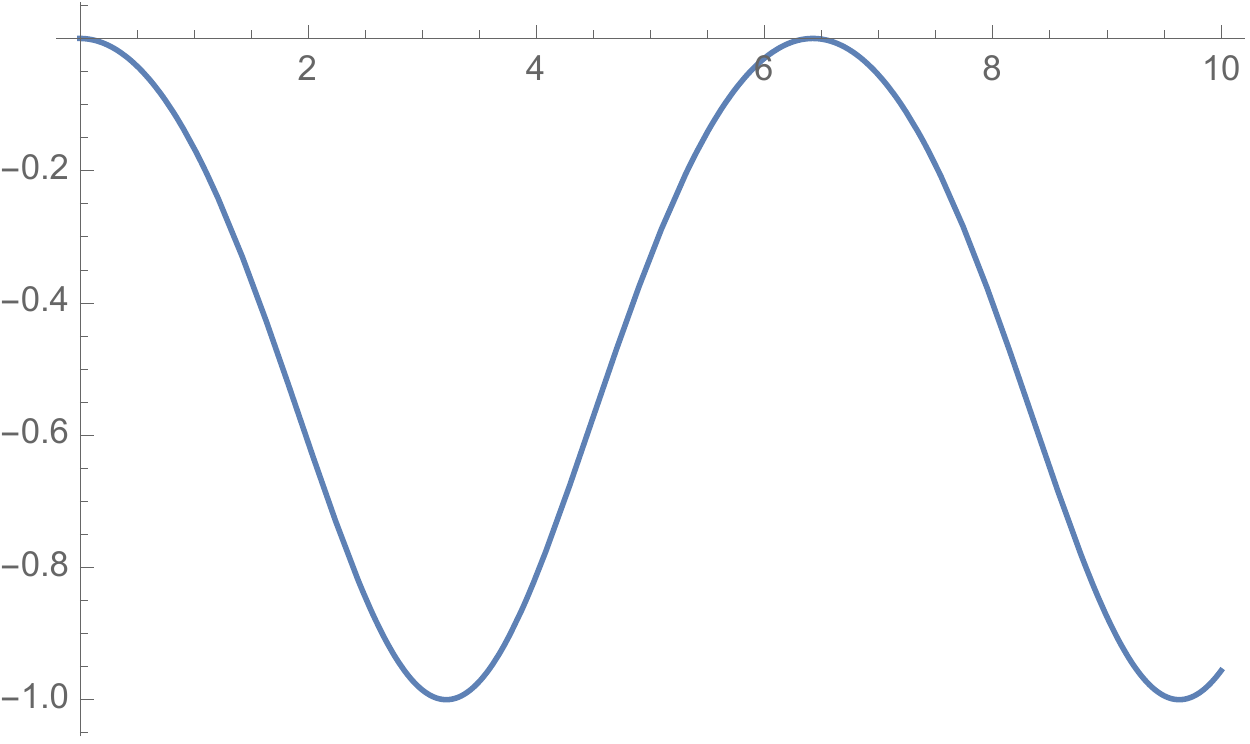}
  \caption{Graph of (\ref{eq:6-1}) with $ c=-1,a=1 $}
\end{figure}

Here the oval part in Figure 1, parametrised by $ (x_{0},\dot{x}_{0}) $:
\begin{figure}[ht]
  \centering
  \includegraphics[scale=0.3]{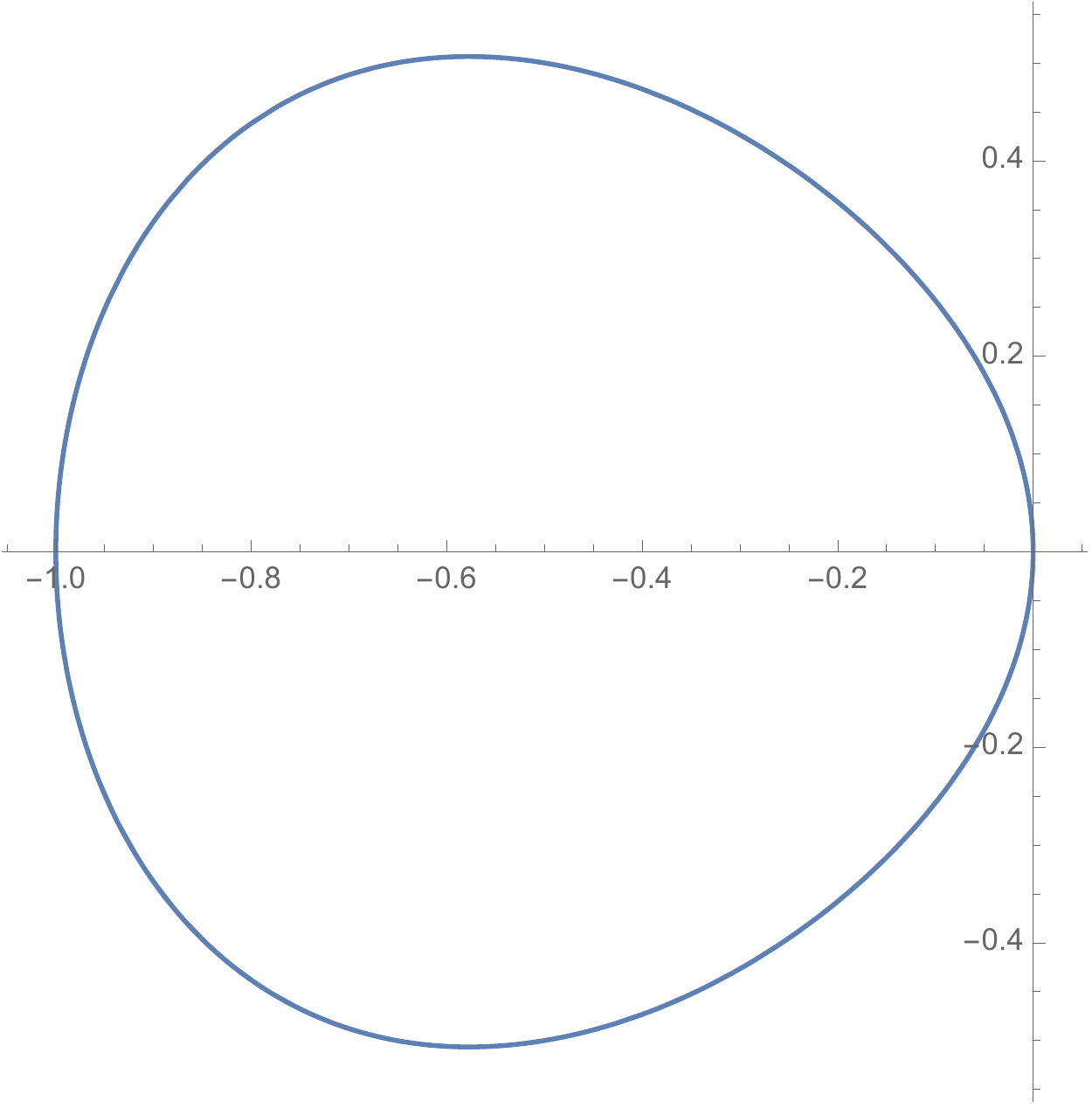}
  \caption{Graph of $ H(y,\eta) =0 $}
\end{figure}

\subsection{The equation for $x_1$: a stochastic Lam\'e's equation}

We now want to solve
\begin{eqnarray}\label{x_1 bis}
\ddot{x}_1(t)=2x_0(t)x_1(t)+\dot{B}(t),\quad x_1(0)=0\quad \dot{x}_1(0)=0
\end{eqnarray} 
which is equivalent to the system of stochastic differential equations
\begin{equation*}\label{x_1 system}
\begin{cases}
dx_1(t) = \xi_1(t)dt, & x_1(0)=0\\
d\xi_1(t)= 2x_0(t)x_1(t)dt+dB(t), & \xi_1(0)=0
\end{cases}
\end{equation*}
with $ x_{0}(t) $ given by (\ref{eq:6-1}). We first investigate the homogeneous equation 
\begin{equation}\label{eq:7}
\ddot{w}(t) - 2x_{0}(t)w(t) =0
\end{equation}
and search for two independent solutions. We observe that according to formula (\ref{eq:6-1}) equation (\ref{eq:7}) can be rewritten as
\begin{equation*}
\ddot{w}(t) - 2\left(c - (a+2c)\text{sn}^{2}\left(\sqrt{\frac{a-c}{6}}t + i_{y},q\right)\right)w(t) =0.
\end{equation*}
It is then equivalent to study equation
\begin{eqnarray}\label{Lame w}
\ddot{u}(t) - \frac{12}{a-c}\left(c - (a+2c)\text{sn}^{2}\left(t,q\right)\right)u(t) =0
\end{eqnarray}
and set
\begin{eqnarray*}
w(t):=u\left(\sqrt{\frac{a-c}{6}}t + i_{y}\right),\quad t\geq 0.
\end{eqnarray*}
The first solution we are going to find is related to the Lam\'e's equation which we now briefly recall. Lam\'e's equation is usually given as
\begin{equation}\label{eq:9}
\ddot{u}(t)  + (h - \nu(\nu+1)q^{2}\text{sn}^{2}(t,q))u(t)=0.
\end{equation}
For fixed $ q $ and $ \nu $ an eigenvalue of (\ref{eq:9}) is a value of $ h $ for which (\ref{eq:9}) has a nontrivial odd or even solution with period $ 2K $ or $ 4K $ where $ K =K(q) = F(\frac{\pi}{2},q)$ (recall equality (\ref{F})). Comparing (\ref{eq:9}) with (\ref{Lame w}) we see that 
\begin{eqnarray}\label{h}
h = \frac{12c}{c-a}= 4 + 4q^{2}
\end{eqnarray}
due to the equality $ q = \sqrt{\frac{2c+a}{c-a}}$. Moreover, since
$t\mapsto\text{sn}(t,q) $ is periodic of period $ 4K $ and $t\mapsto
\text{sn}^{2}(t,q) $ is periodic of period $ 2K$, we conclude that $h$
given in (\ref{h}) is an eigenvalue of (\ref{eq:9}) corresponding to
case (8) at pag. X
in Arscott and Khabaza \cite{AK} and hence
\begin{eqnarray}\label{w_1}
\displaystyle u_{1}(t) = \text{sn}(t,q)\text{cn}(t,q) \text{dn}(t,q),\quad t\geq 0 
\end{eqnarray}
is the first solution of  (\ref{Lame w}) we are looking for. It is a special Lam\'e's polynomial of order three satisfying $ u_{1}(0) =0$ and $u_{1}'(0) =1 $.

\begin{figure}[ht]
  \centering
  \includegraphics[scale=0.6]{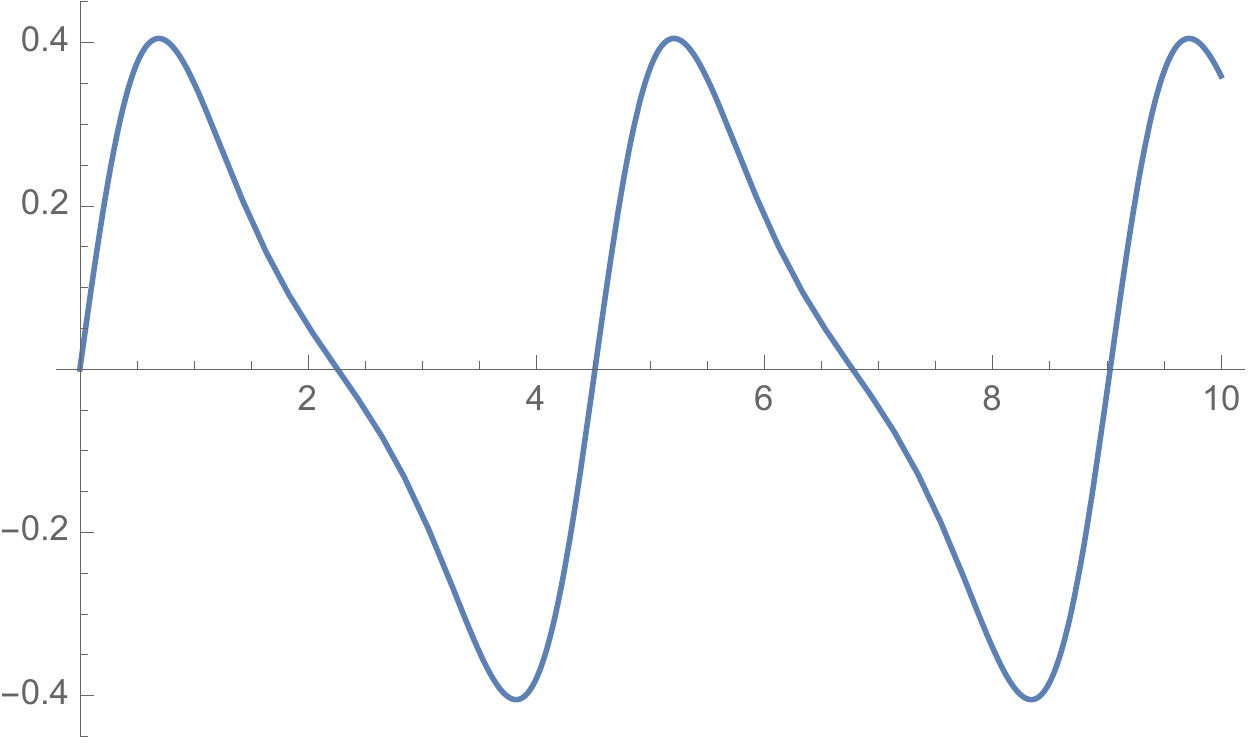}
  \caption{Graph of $u_1$ with $ q =2/\sqrt{5} $}
\end{figure}

We now need another independent solution. It is a very elementary fact that if $ u_{1} $ solves the equation $
\ddot{u}(t) + \alpha u(t) =0 $, then $ u_{2}(t): = u_{1}(t)\cdot \int^{t}\frac{ds}{u_{1}^{2}(s)}$
solves the same equation. Since 
\begin{equation*}\label{eq:10}
\int\frac{dt}{ \left[\text{sn}(t,q)
\text{cn}(t,q) \text{dn}(t,q)\right]^{2}} = \frac{C(t,q)}{D(t,q)}
\end{equation*}
with
\begin{eqnarray*}
C(t,q): &=& -\text{dn}(t,q)\biggl[ -1 + q^2 + \biggl(2 + q^2(-5 + (3-2q^2)q^2)\biggr)\text{cn}^{2}(t,q)\\
&&+2q^2\biggl(1 + (-1+q^2)q^2\biggr)\text{cn}^{4}(t,q)\\
&&+\biggl((2 -q^2)(-1+q^2)x + 2(q^4 - q^2 +1)\mathcal{E}(t,q)\biggr)\\
&&\times\text{sn}(t,q)\text{cn}(t,q) \text{dn}(t,q)\biggr]
\end{eqnarray*}
and
\begin{eqnarray*}
D(t,q) := (-1 + q^2)^2 \text{sn}(t,q)
\text{cn}(t,q) \text{dn}^{2}(t,q)
\end{eqnarray*}
we get that 
\begin{equation}\label{eq:13}
u_{2}(t) = \mathcal{C}(t,q)+\mathcal{D}(t,q)u_{1}(t,q)
\end{equation}
is a second independent solution of (\ref{Lame w}). Here
\begin{eqnarray*}
\mathcal{C}(t,q) &:=& \alpha_{0}(q) + \alpha_{1}(q)\text{cn}^{2}(t,q) + \alpha_{2}(q)\text{cn}^{4}(t,q)\\
\mathcal{D}(t,q) &:=& \beta_{0}(q)t + \beta_{1}(q)\mathcal{E}(t,q)
\end{eqnarray*}
and
\begin{eqnarray*}
\alpha_{0}(q) &:=& -1 + q^{2}\\
\alpha_{1}(q) &:=& -2q^{6} + 3q^{4}  -5q^{2} +2 \\
\alpha_{2}(q) &:=& 2q^{2}(q^{4} -q^{2} +1) \\
\beta_{0}(q) &:=& -q^{4} + 3q^{2} -2 \\
\beta_{1}(q) &:=& 2(q^{4} - q^{2} +1).
\end{eqnarray*}

\begin{figure}[ht]
  \centering
  \includegraphics[scale=0.6]{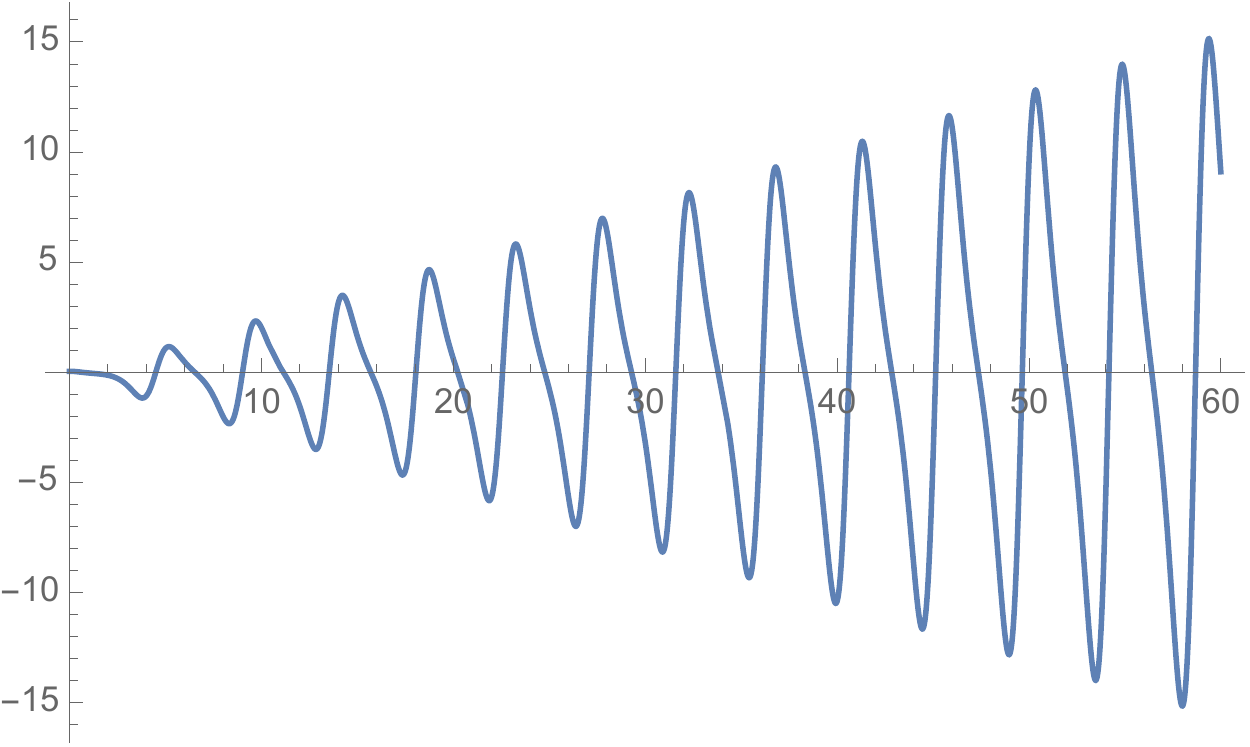}
  \caption{Graph of $u_{2}$ with $ q =2/\sqrt{5} $}
\end{figure}

The coefficient of $t$ in (\ref{eq:13}) is given by 
\begin{eqnarray*}
\mu(q) :=\beta_{0}(q) + \beta_{1}(q)\frac{E(q)}{K(q)}
\end{eqnarray*}
with $ E(q) $ denotes the complete elliptic integral of the second kind. This coefficient behaves like this when $ 0<q<1 $:

\begin{figure}[H]
  \centering
  \includegraphics[scale=0.6]{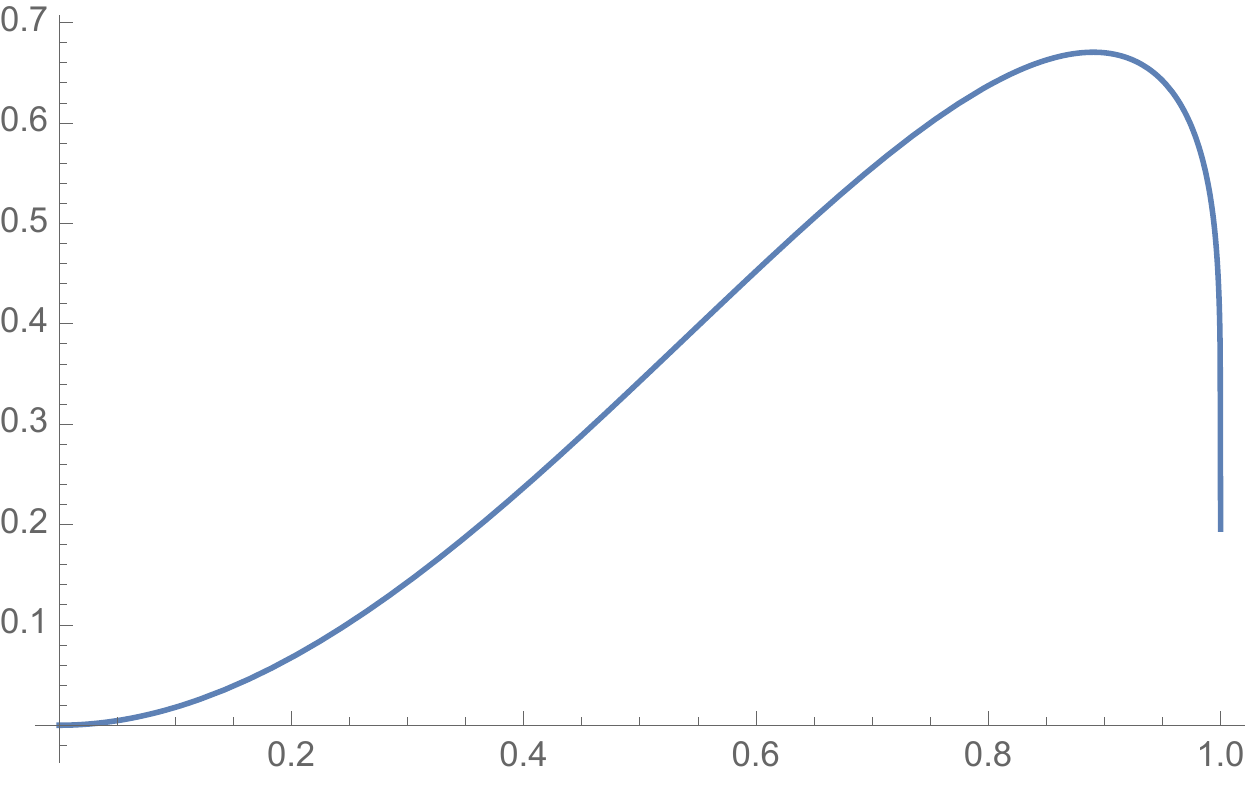}
  \caption{Graph of $\mu(q)$ with $ q \in (0,1) $}
\end{figure}

Going back to equation (\ref{eq:7}) we have found  two linearly independent solutions:
\begin{eqnarray*}
w_{1}(t) = u_{1}\left(\sqrt{\frac{a-c}{6}}t + i_{y}\right)\quad\mbox{ and }\quad w_{2}(t) = u_{2}\left(\sqrt{\frac{a-c}{6}}t + i_{y}\right).
\end{eqnarray*}
It is easy to verify that the Wronskian determinant of $(u_{1},u_{2})$ is $ -(1 - q^{2})^{2} \neq 0 $ and that of $
(w_{1},w_{2}) $ is $ -\sqrt{\frac{a-c}{6}}(1 - q^{2})^{2}$
Without loss of generality multiplying $ w_{1} $ and $w_{2}$ by
suitable constants we may assume that their  Wronskian determinant is $ 1 $.
Moreover, since $u_{1}$ is periodic of period $2K$ we get that $w_{1}$
is periodic of period $\displaystyle 2\sqrt{6/(a-c)}K $.  A simple
application of standard Floquet-Lyapunov results, see e.g Yakubovich and Starzhinskii \cite{YS}
page 97, tells us that (\ref{eq:9}) has one periodic solution (here $
w_{1} $) and it is unstable, due to a double eigenvalue in the
monodromy matrix. We omit the trivial details. 

\begin{theorem}\label{theorem x_1}
Equation (\ref{x_1 bis}) has a unique global strong solution adapted to the Brownian filtration $\{\mathcal{F}^B_t\}_{t\geq 0}$. The solution is a continuous Gaussian process which can be explicitly represented as 
\begin{eqnarray*}
x_1(t)&=&w_2(t)\int_0^tw_1(s)dB(s)-w_1(t)\int_0^tw_2(s)dB(s)
\end{eqnarray*}
or equivalently
\begin{eqnarray}\label{x_1 representation}
x_1(t)=\int_0^t\mathcal{K}(t,s)dB(s)\quad\mbox{ with }\quad\mathcal{K}(t,s):=w_2(t)w_1(s)-w_1(t)w_2(s).
\end{eqnarray}
\end{theorem}

\begin{proof}
It is clear from (\ref{x_1 representation}) and the basic properties of the Wiener integral  that $\{x_1(t)\}_{t\geq 0}$ is a continuous Gaussian process adapted to the Brownian filtration $\{\mathcal{F}^B_t\}_{t\geq 0}$. We have to verify that it solves equation (\ref{x_1 bis}) (uniqueness follows from the linearity of the equation). We first observe that since $\mathcal{K}(t,t)=0$ and $B(0)=0$ we obtain integrating by parts
\begin{eqnarray*}
x_1(t)&=&-\int_{0}^t\partial_2\mathcal{K}(t,s)B(s)ds
\end{eqnarray*}
(note that $w_1$ and $w_2$ are real analytic and that $\partial_i$ stands for differentiation with respect to the $i$-th variable). This gives
\begin{eqnarray*}
\dot{x}_1(t)=-\partial_2\mathcal{K}(t,t)B(t)-\int_{0}^t\partial_1\partial_2\mathcal{K}(t,s)B(s)ds
\end{eqnarray*}
and hence $x_1(0)=\dot{x}_1(0)=0$. Now, let $f\in C([0,+\infty[)\cap L^2([0,+\infty[)$ and consider the random element
\begin{eqnarray*}
\text{Exp}(f):=\exp\left\{\int_0^{+\infty}f(t)dB(t)-\frac{1}{2}\int_0^{+\infty}f^2(t)dt\right\}.
\end{eqnarray*}
Using It\^o's formula and isometry one sees that
\begin{eqnarray*}
\mathbb{E}\left[x_1(t)\text{Exp}(f)\right]=\int_0^t\mathcal{K}(t,s)f(s)ds.
\end{eqnarray*} 
A direct verification exploiting the fact that $w_1$ and $w_2$ are solutions of (\ref{eq:7}) with Wronskian equal to one shows that the function $t\mapsto \int_0^t\mathcal{K}(t,s)f(s)ds$ solves
\begin{eqnarray*}
\ddot{z}(t)=2x_0(t)z(t)+f(t),\quad z(0)=0\quad\dot{z}(0)=0
\end{eqnarray*}
which is equivalent to
\begin{eqnarray*}
\dot{z}(t)=2\int_0^tx_0(s)z(s)ds+\int_0^tf(s)ds,\quad z(0)=0.
\end{eqnarray*}
Therefore, since $t\mapsto x_1(t)$ is almost surely continuously differentiable with $\dot{x}_1(t)\in L^2(\Omega,\mathcal{F}^B,\mathbb{P})$ (here $\mathcal{F}^B:=\cup_{t\geq 0}\mathcal{F}_t^B\subseteq\mathcal{F}$), we get 
\begin{eqnarray*}
\mathbb{E}[\dot{x}_1(t)\text{Exp}(f)]&=&\frac{d}{dt}\mathbb{E}[x_1(t)\text{Exp}(f)]\\
&=&\frac{d}{dt}\int_0^t\mathcal{K}(t,s)f(s)ds\\
&=&2\int_0^tx_0(s)\left(\int_0^s\mathcal{K}(s,r)f(r)dr\right)ds+\int_0^tf(s)ds\\
&=&2\int_0^tx_0(s)\mathbb{E}[x_1(s)\text{Exp}(f)]ds+\mathbb{E}[B(t)\text{Exp}(f)]\\
&=&\mathbb{E}\left[\left(2\int_0^tx_0(s)x_1(s)ds+B(t)\right)\text{Exp}(f)\right].
\end{eqnarray*}
Since the set $\left\{\text{Exp}(f), f\in C([0,+\infty[)\cap L^2([0,+\infty) \right\}$ is total in $L^2(\Omega,\mathcal{F}^B,\mathbb{P})$ we conclude that
\begin{eqnarray*}
\dot{x}_1(t)=2\int_0^tx_0(s)x_1(s)ds+B(t)\quad\mbox{ almost surely}.
\end{eqnarray*}
The proof is now complete.
\end{proof}

\begin{example}
If we consider
\begin{eqnarray*} 
\ddot{x}_0(t) = x_0(t)^{2} -1/3,\quad  x_0(0)=0\quad \dot{x}_0(0) =0
\end{eqnarray*}
we see that
\begin{eqnarray*}
x_0(t) = - (6 - 3\emph{ \text{sn}}^{2}(t,1/2))/2.
\end{eqnarray*}
With such choice of the coefficients equation (\ref{x_1 bis}) becomes
\begin{equation}\label{eq:33}
\ddot{x}_1(t)= - (6 - 3\emph{ \text{sn}}^{2}(t,1/2))x_1(t)+\dot{B}(t),\quad x_1(0) =0\quad \dot{x}_1(0) =0.
\end{equation}
The picture below gives a clear idea of the oscillatory nature of the
solution process (\ref{eq:33}).

\begin{figure}[H]
  \centering
  \includegraphics[scale=0.7]{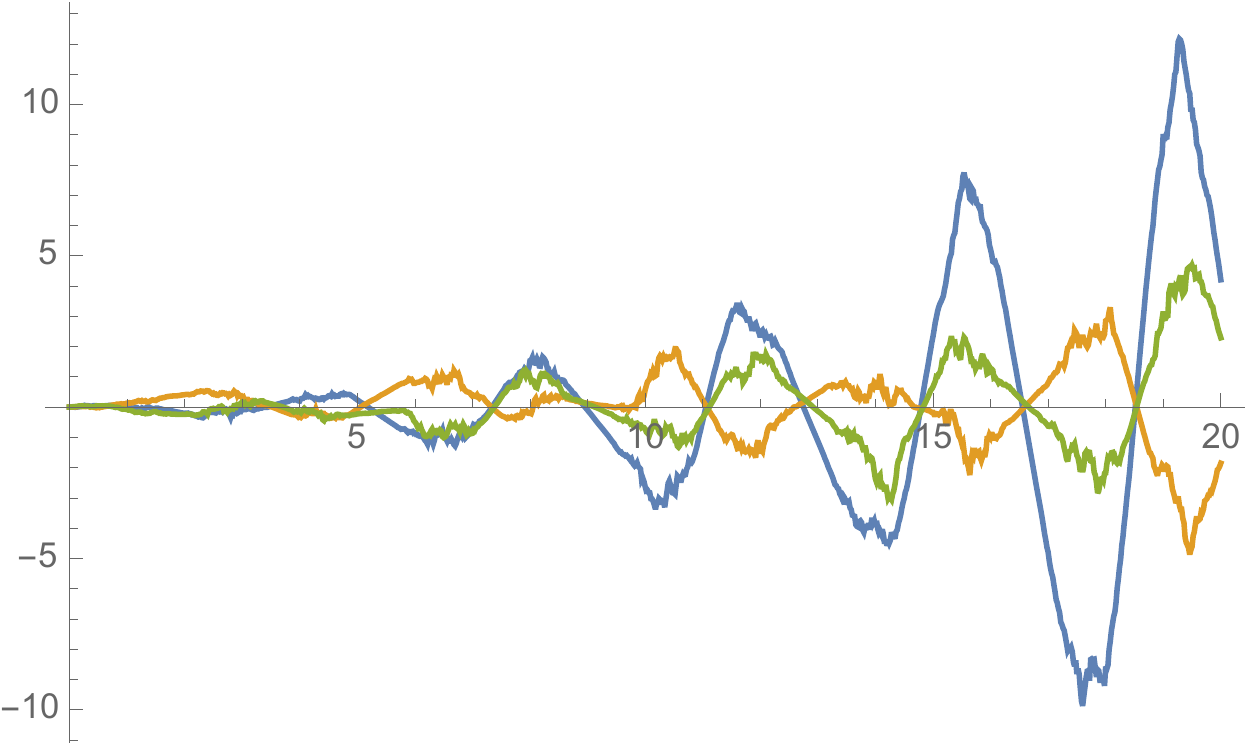}
  \caption{Graph of three paths of process (\ref{eq:33})}
\end{figure}
\end{example}

\subsection{The equations for $x_n$ with $n\geq 2$: Lam\'e's equations with random coefficients}

We now want to solve
\begin{eqnarray}\label{x_n bis}
\ddot{x}_n(t)=2x_0(t)x_n(t)+\sum_{j=1}^{n-1}x_j(t)x_{n-j}(t),\quad x_n(0)=0\quad \dot{x}_n(0)=0
\end{eqnarray} 
which is equivalent to the system of differential equations with random coefficients
\begin{equation*}
\begin{cases}
\dot{x}_n(t) = \xi_n(t), & x_n(0)=0\\
\dot{\xi}_n(t)= 2x_0(t)x_n(t)+\sum_{j=1}^{n-1}x_j(t)x_{n-j}(t), & \xi_n(0)=0
\end{cases}
\end{equation*}
with $ x_{0}(t) $ given by (\ref{eq:6-1}). We remark that the non homogeneous term $\sum_{j=1}^{n-1}x_j(t)x_{n-j}(t)$ depends on the random processes $\{x_m(t)\}_{t\geq 0}$ for $m<n$. Therefore, equation (\ref{x_n bis}) is described inductively by solving the linear equations associated to lower terms in the expansion (\ref{series}). We have the following.

\begin{theorem}\label{theorem x_n}
For every $n\geq 2$ equation (\ref{x_n bis})  has a unique global strong solution adapted to the Brownian filtration $\{\mathcal{F}^B_t\}_{t\geq 0}$. The solution can be explicitly represented as 
\begin{eqnarray*}
x_n(t)&=&w_2(t)\int_0^tw_1(s)\left(\sum_{j=1}^{n-1}x_j(s)x_{n-j}(s)\right)ds\\
&&-w_1(t)\int_0^tw_2(s)\left(\sum_{j=1}^{n-1}x_j(s)x_{n-j}(s)\right)ds
\end{eqnarray*}
or equivalently
\begin{eqnarray*}
x_n(t)&=&\int_0^t\mathcal{K}(t,s)\left(\sum_{j=1}^{n-1}x_j(s)x_{n-j}(s)\right)ds.
\end{eqnarray*}
\end{theorem}

\begin{proof}
The proof is obtained via straightforward modifications of the proof of Theorem \ref{theorem x_1}.
\end{proof}

\section{Global behavior of the truncated expansion}

The next result provides some estimates on the random behavior of the functions $\{x_n(t)\}_{t\geq 0}$ for $n\geq 1$.  

\begin{proposition}\label{main theorem}
For any $T>0$ and $n\geq 1$ we have
\begin{eqnarray*}
&&\mathbb{P}\left(|x_n(t)|\leq\frac{\gamma_n(t)}{\sigma^n}\mbox{ for all $t\in [0,T]$}\right)\\
&&\geq 1-\sigma 2^{(n-2)^+}\sqrt{2/\pi}\left(\Vert w_1\Vert_{L^2([0,T])}+\Vert w_2\Vert_{L^2([0,T])}\right)
\end{eqnarray*} 
where $(n-2)^+:=\max\{n-2,0\}$ while $\{\gamma_n\}_{n\geq 1}$ is defined recursively as 
\begin{eqnarray}\label{Gamma_1}
\gamma_1(t):=|w_1(t)|+|w_2(t)|,\quad t\geq 0
\end{eqnarray}
and for $n\geq 2$,
\begin{eqnarray*}
\gamma_n(t):=\int_0^t|K(t,s)|\left(\sum_{j=1}^{n-1}\gamma_j(s)\gamma_{n-j}(s)\right)ds,\quad t\geq 0.
\end{eqnarray*}
\end{proposition}

\begin{proof}
We will prove the theorem by induction dividing the proof in two steps.\\

\noindent\textbf{Step One: $n=1$}. We recall that
\begin{eqnarray*}
x_1(t)&=&\int_0^tK(t,s)dB(s)\\
&=&w_2(t)\int_0^tw_1(s)dB(s)+w_1(t)\int_0^tw_2(s)dB(s).
\end{eqnarray*}
We fix a positive constant $T$ and observe that
\begin{eqnarray*}
|x_1(t)|&\leq& |w_2(t)|\left|\int_0^tw_1(s)dB(s)\right|+|w_1(t)|\left|\int_0^tw_2(s)dB(s)\right|\\
&\leq& |w_2(t)|\sup_{t\in [0,T]}\left|\int_0^tw_1(s)dB(s)\right|+|w_1(t)|\sup_{t\in [0,T]}\left|\int_0^tw_2(s)dB(s)\right|.
\end{eqnarray*}
We now denote
\begin{eqnarray*}
A_1:=\left\{\omega\in\Omega: \sup_{t\in [0,T]}\left|\int_0^tw_1(s)dB(s)\right|\leq1/\sigma\right\}
\end{eqnarray*}
and
\begin{eqnarray*}
A_2:=\left\{\omega\in\Omega: \sup_{t\in [0,T]}\left|\int_0^tw_2(s)dB(s)\right|\leq1/\sigma\right\}.
\end{eqnarray*}
On the set $A_1\cap A_2$ the inequality
\begin{eqnarray*}
|x_1(t)|\leq \frac{|w_1(t)|+|w_2(t)|}{\sigma}\quad\mbox{ for all $t\in [0,T]$}
\end{eqnarray*}
holds true; we can therefore write recalling (\ref{Gamma_1}) that
\begin{eqnarray}\label{A}
\mathbb{P}\left(|x_1(t)|\leq\frac{\gamma_1(t)}{\sigma}\mbox{ for all $t\in [0,T]$}\right)&\geq&\mathbb{P}(A_1\cap A_2)\nonumber\\
&=&1-\mathbb{P}\left(A_1^c\cup A_2^c\right)\nonumber\\
&\geq&1-\mathbb{P}\left(A_1^c\right)-\mathbb{P}\left(A_2^c\right).
\end{eqnarray}
Now, according to Doob's maximal inequality for $i=1,2$ we  have
\begin{eqnarray}\label{B}
\mathbb{P}\left(A_i^c\right)&=&\mathbb{P}\left(\sup_{t\in [0,T]}\left|\int_0^tw_i(s)dB(s)\right|>1/\sigma\right)\nonumber\\
&\leq&\sigma\mathbb{E}\left[\left|\int_0^Tw_i(s)dB(s)\right|\right]\nonumber\\
&=&\sigma\sqrt{2/\pi}\Vert w_i\Vert_{L^2([0,T])}.
\end{eqnarray}
Hence, combining the upper bound (\ref{B}) with the lower bound (\ref{A}) we conclude that
\begin{eqnarray*}
\mathbb{P}\left(|x_1(t)|\leq\frac{\gamma_1(t)}{\sigma}\mbox{ for all $t\in [0,T]$}\right)\geq 1-\sigma\sqrt{2/\pi}\left(\Vert w_1\Vert_{L^2([0,T])}+\Vert w_2\Vert_{L^2([0,T])}\right).
\end{eqnarray*}

\noindent\textbf{Step two: $n\geq 2$}. We now assume the statement to be true for any $i\leq n-1$ and prove it for $i=n$. According to the representation 
\begin{eqnarray*}
x_n(t)=\int_0^tK(t,s)\left(\sum_{j=1}^{n-1}x_j(s)x_{n-j}(s)\right)ds,\quad t\geq 0
\end{eqnarray*}
we can bound $|x_n(t)|$ as follows
\begin{eqnarray*}
|x_n(t)|\leq\int_0^t|K(t,s)|\left(\sum_{j=1}^{n-1}|x_j(s)||x_{n-j}(s)|\right)ds.
\end{eqnarray*}
We now denote for $i\leq n-1$
\begin{eqnarray*}
A_i:=\left\{\omega\in\Omega: |x_i(t)|\leq\frac{\gamma_i(t)}{\sigma^i}\mbox{ for all $t\in [0,T]$}\right\}
\end{eqnarray*}
and observe that according to the inductive hypothesis
\begin{eqnarray*}
\mathbb{P}(A_i)\geq 1-\sigma 2^{(i-2)^+}\sqrt{2/\pi}\left(\Vert w_1\Vert_{L^2([0,T])}+\Vert w_2\Vert_{L^2([0,T])}\right)
\end{eqnarray*}
or equivalently,
\begin{eqnarray}\label{complement}
\mathbb{P}(A_i^c)\leq \sigma 2^{(i-2)^+}\sqrt{2/\pi}\left(\Vert w_1\Vert_{L^2([0,T])}+\Vert w_2\Vert_{L^2([0,T])}\right).
\end{eqnarray}
We also note that on the set $A_1\cap\cdot\cdot\cdot\cap A_{n-1}$ we have
\begin{eqnarray*}
|x_n(t)|&\leq&\int_0^t|K(t,s)|\left(\sum_{j=1}^{n-1}|x_j(s)||x_{n-j}(s)|\right)ds\\
&\leq&\int_0^t|K(t,s)|\left(\sum_{j=1}^{n-1}\frac{\gamma_j(s)}{\sigma^j}\frac{\gamma_{n-j}(s)}{\sigma^{n-j}}\right)ds\\
&=&\frac{\gamma_n(t)}{\sigma^n}
\end{eqnarray*}
for all $t\in [0,T]$. Therefore,
\begin{eqnarray*}
&&\mathbb{P}\left(|x_n(t)|\leq\frac{\gamma_n(t)}{\sigma^n}\mbox{ for all $t\in [0,T]$}\right)\\
&&\geq\mathbb{P}\left(A_1\cap\cdot\cdot\cdot\cap A_{n-1}\right)\\
&&=1-\mathbb{P}\left(A_1^c\cup\cdot\cdot\cdot\cup A_{n-1}^c\right)\\
&&\geq 1-\mathbb{P}\left(A_1^c\right)-\cdot\cdot\cdot-\mathbb{P}\left(A_{n-1}^c\right)\\
&&=1-\sum_{i=1}^{n-1}\mathbb{P}(A_i^c)\\
&&\geq1-\sigma\sqrt{2/\pi}\left(\Vert w_1\Vert_{L^2([0,T])}+\Vert w_2\Vert_{L^2([0,T])}\right)\sum_{i=1}^{n-1} 2^{(i-2)^+}\\
&&=1-\sigma 2^{(n-2)^+}\sqrt{2/\pi}\left(\Vert w_1\Vert_{L^2([0,T])}+\Vert w_2\Vert_{L^2([0,T])}\right)
\end{eqnarray*}
where in the last inequality we utilized the bound (\ref{complement}). The proof is complete. 
\end{proof}

We now prove a lower bound for the probability that the $n$-th order approximation of the \emph{virtual} solution process $\{x(t)\}_{t\geq 0}$ stays close to the solution of the deterministic equation during a given time interval $[0,T]$.
\begin{theorem}\label{corollary}
For $ n\geq 1$ we let
\begin{eqnarray*}
X^{\sigma}_n(t):=x_0(t)+\sigma x_1(t)+\cdot\cdot\cdot+\sigma^n x_n(t),\quad t\geq 0.
\end{eqnarray*}
Then, for any $T>0$ we have
\begin{eqnarray*}
&&\mathbb{P}\left(|X^{\sigma}_n(t)-x_0(t)|\leq\Gamma_n(t)\mbox{ for all $t\in [0,T]$}\right)\\
&&\geq 1-\sigma 2^{(n-1)^+}\sqrt{2/\pi}\left(\Vert w_1\Vert_{L^2([0,T])}+\Vert w_2\Vert_{L^2([0,T])}\right)
\end{eqnarray*}
where
\begin{eqnarray*}
\Gamma_n(t):=\sum_{i=1}^n\gamma_i(t),\quad t\geq 0
\end{eqnarray*}
and $\{\gamma_n\}_{n\geq 1}$ is the sequence of functions defined in Theorem \ref{main theorem}.
\end{theorem}

\begin{proof}
We proceed as before. For $i\leq n$ we introduce the events
\begin{eqnarray*}
A_i:=\left\{\omega\in\Omega: |x_i(t)|\leq\frac{\gamma_i(t)}{\sigma^i}\mbox{ for all $t\in [0,T]$}\right\}
\end{eqnarray*}
and observe that according to Theorem \ref{main theorem} we have
\begin{eqnarray*}
\mathbb{P}(A_i^c)\leq \sigma 2^{(i-2)^+}\sqrt{2/\pi}\left(\Vert w_1\Vert_{L^2([0,T])}+\Vert w_2\Vert_{L^2([0,T])}\right).
\end{eqnarray*}
We also note that on the set $A_1\cap\cdot\cdot\cdot\cap A_{n}$ we have
\begin{eqnarray*}
|X^{\sigma}_n(t)-x_0(t)|&\leq&\sum_{i=1}^n\sigma^i|x_i(t)|\\
&\leq&\sum_{i=1}^n\gamma_i(t)\\
&=&\Gamma_n(t)
\end{eqnarray*}
for all $t\in [0,T]$. Therefore,
\begin{eqnarray*}
&&\mathbb{P}\left(|X^{\sigma}_n(t)-x_0(t)|\leq \Gamma_n(t)\mbox{ for all $t\in [0,T]$}\right)\\
&&\geq\mathbb{P}\left(A_1\cap\cdot\cdot\cdot\cap A_{n}\right)\\
&&=1-\mathbb{P}\left(A_1^c\cup\cdot\cdot\cdot\cup A_{n}^c\right)\\
&&\geq1-\mathbb{P}\left(A_1^c\right)-\cdot\cdot\cdot-\mathbb{P}\left(A_{n}^c\right)\\
&&=1-\sum_{i=1}^{n}\mathbb{P}(A_i^c)\\
&&\geq 1-\sigma\sqrt{2/\pi}\left(\Vert w_1\Vert_{L^2([0,T])}+\Vert w_2\Vert_{L^2([0,T])}\right)\sum_{i=1}^{n} 2^{(i-2)^+}\\
&&=1-\sigma 2^{(n-1)^+}\sqrt{2/\pi}\left(\Vert w_1\Vert_{L^2([0,T])}+\Vert w_2\Vert_{L^2([0,T])}\right).
\end{eqnarray*}
\end{proof}

\section{Stochastic equation driven by a \emph{bounded} martingale}

In this section we consider the second order differential equation 
\begin{eqnarray}\label{SDE bounded}
\ddot{\chi}(t)=\chi(t)^2-\mathcal{B}+\sigma\dot{Z}_t,\quad \chi(0)=y\quad \dot{\chi}(0)=\eta
\end{eqnarray} 
where now $\{Z_t\}_{t\geq 0}$ is a bounded continuous martingale starting at zero. The next theorem shows that in this case the series
\begin{eqnarray}\label{series bounded}
\chi(t):=\chi_0(t)+\sum_{n\geq 1}\sigma^n\chi_n(t)
\end{eqnarray}
converges almost surely for all $t$ in a suitably small interval. We observe that $\chi_0(t)=x_0(t)$ for all $t\geq 0$ since the equation they solve is not affected by choice of the noise.

\begin{theorem}
Let $\{Z_t\}_{t\geq 0}$ satisfy for all $t\geq 0$ the condition $Z(t)\in L^{\infty}(\Omega,\mathcal{F},\mathbb{P})$. Then, there exists $T_{\sigma}>0$ such that the series (\ref{series bounded}) converges almost surely for any $t\in [0,T_{\sigma}]$. More precisely, the uniform bound
\begin{eqnarray}\label{second}
\sup_{t\in [0,T_{\sigma}]}|x(t)-x_0(t)|\leq \frac{1}{2\mathcal{N}(T_{\sigma})}\quad\mbox{ almost surely}
\end{eqnarray}
holds with 
\begin{eqnarray*}
\mathcal{N}(T):=\Vert w_2\Vert_{L^{\infty}([0,T])}\Vert w_1\Vert_{L^{1}([0,T])}+\Vert w_1\Vert_{L^{\infty}([0,T])}\Vert w_2\Vert_{L^{1}([0,T])}.
\end{eqnarray*}
\end{theorem}

\begin{proof}
We start as before observing that
\begin{eqnarray*}
\chi_1(t)&=&\int_0^t\mathcal{K}(t,s)dZ(s)\\
&=&-\int_0^t\partial_2\mathcal{K}(t,s)Z(s)ds
\end{eqnarray*}
(recall the kernel $\mathcal{K}$ defined in (\ref{x_1 representation})). This implies
\begin{eqnarray*}
|\chi_1(t)|&\leq&\int_0^t|\partial_2\mathcal{K}(t,s)||Z(s)|ds\\
&=&\sup_{s\in [0,t]}\Vert Z(s)\Vert_{L^{\infty}(\Omega)}\int_0^t|\partial_2\mathcal{K}(t,s)|ds\\
&=&\sup_{s\in [0,t]}\Vert Z(s)\Vert_{L^{\infty}(\Omega)}\left(|\dot{w}_2(t)|\int_0^t|w_1(s)|ds+|\dot{w}_1(t)|\int_0^t|w_2(s)|ds\right)\\
&\leq&\sup_{s\in [0,t]}\Vert Z(s)\Vert_{L^{\infty}(\Omega)}\left(\Vert \dot{w}_2\Vert_{L^{\infty}([0,t])}\Vert w_1\Vert_{L^{1}([0,t])}+\Vert \dot{w}_1\Vert_{L^{\infty}([0,t])}\Vert w_2\Vert_{L^{1}([0,t])}\right)\\
&=&\mathcal{M}(t)
\end{eqnarray*}
where to ease the notation we set
\begin{eqnarray*}
\mathcal{M}(t):=\sup_{s\in [0,t]}\Vert Z(s)\Vert_{L^{\infty}(\Omega)}\left(\Vert \dot{w}_2\Vert_{L^{\infty}([0,t])}\Vert w_1\Vert_{L^{1}([0,t])}+\Vert \dot{w}_1\Vert_{L^{\infty}([0,t])}\Vert w_2\Vert_{L^{1}([0,t])}\right).
\end{eqnarray*}
The first step is to prove by induction that for any fixed $T>0$ we have
\begin{eqnarray}\label{induction 1}
|\chi_n(t)|\leq c_n\mathcal{M}(T)^n\mathcal{N}(T)^{n-1}\quad\mbox{ almost surely for all $n\geq 1$ and $t\in [0,T]$}
\end{eqnarray}
where $\{c_n\}_{n\geq 1}$ denotes the sequence of \emph{Catalan numbers} which are defined recursively as
\begin{eqnarray*}
c_1=1,\quad c_2=1,\quad c_n:=\sum_{j=1}^{n-1}c_jc_{n-j}.
\end{eqnarray*}
Inequality (\ref{induction 1}) is trivially true for $n=1$ (note that the function $t\mapsto \mathcal{M}(t)$ is increasing). We now assume the property to be true for all $j\leq n-1$; then, 
\begin{eqnarray*}
|\chi_n(t)|&=&\left|\int_0^t\mathcal{K}(t,s)\left(\sum_{j=1}^{n-1}\chi_j(s)\chi_{n-j}(s)\right)ds\right|\\
&\leq&\int_0^t|\mathcal{K}(t,s)|\left(\sum_{j=1}^{n-1}|\chi_j(s)| |\chi_{n-j}(s)|\right)ds\\
&\leq&\int_0^t|\mathcal{K}(t,s)|\left(\sum_{j=1}^{n-1} c_j\mathcal{M}(T)^j \mathcal{N}(T)^{j-1} c_{n-j}\mathcal{M}(T)^{n-j} \mathcal{N}(T)^{n-j-1}\right)ds\\
&=&\mathcal{M}(T)^n\mathcal{N}(T)^{n-2}\int_0^t|\mathcal{K}(t,s)|\left(\sum_{j=1}^{n-1} c_jc_{n-j}\right)ds\\
&\leq&c_n\mathcal{M}(T)^n\mathcal{N}(T)^{n-1}
\end{eqnarray*}
completing the proof of (\ref{induction 1}). Therefore, recalling that
\begin{eqnarray*}
\sum_{n\geq 1}c_ny^n=\frac{1-\sqrt{1-4y}}{2},\quad y\in [0,1/4]
\end{eqnarray*}
we can write for $t\in [0,T]$ that
\begin{eqnarray}\label{q}
|\chi(t)-\chi_0(t)|&\leq&\sum_{n\geq 1}\sigma^n|\chi_n(t)|\nonumber\\
&\leq&\sum_{n\geq 1}\sigma^nc_n\mathcal{M}(T)^n \mathcal{N}(T)^{n-1}\nonumber\\
&=&\frac{1}{\mathcal{N}(T)}\sum_{n\geq 1}c_n\sigma^n\mathcal{M}(T)^n\mathcal{N}(T)^{n}\nonumber\\
&=&\frac{1-\sqrt{1-4\sigma \mathcal{M}(T)\mathcal{N}(T)}}{2\mathcal{N}(T)}
\end{eqnarray}
provided that
\begin{eqnarray*}
\mathcal{M}(T)\mathcal{N}(T)\leq\frac{1}{4\sigma}.
\end{eqnarray*}
Hence, since $\mathcal{M}(0)=\mathcal{N}(0)=0$ and the functions $T\mapsto \mathcal{M}(T)$ and $T\mapsto \mathcal{N}(T)$ are increasing, continuous and unbounded (by the properties of $w_1$ and $w_2$), we deduce that the equation
\begin{eqnarray*}
\mathcal{M}(T)\mathcal{N}(T)= \frac{1}{4\sigma}
\end{eqnarray*}
has a unique solution $T_{\sigma}>0$ and that $\mathcal{M}(T)\mathcal{N}(T)\leq\frac{1}{4\sigma}$ for all $T\leq T_{\sigma}$. Therefore, choosing $T=T_{\sigma}$ in (\ref{q}) we obtain
\begin{eqnarray*}
\sup_{t\in [0,T_{\sigma}]}|\chi(t)-\chi_0(t)|\leq \frac{1}{2\mathcal{N}(T_{\sigma})}\quad\mbox{ almost surely}.
\end{eqnarray*}
\end{proof}

\begin{example}
We may choose
\begin{eqnarray*}
Z(t):=\sin(B(t))e^{t},\quad t\geq 0
\end{eqnarray*} 
where $\{B(t)\}_{t\geq 0}$ is a one dimensional standard Brownian motion. The process $\{Z(t)\}_{t\geq 0}$ is clearly continuous and starts at zero; moreover, using the It\^o formula one verifies immediately the martingale property. In this case we have
\begin{eqnarray*}
\sup_{s\in [0,t]}\Vert Z(s)\Vert_{L^{\infty}(\Omega)}=e^t,\quad t\geq 0.
\end{eqnarray*}
\end{example}

\appendix

\section{Appendix}\label{appendix}
Here we collect very briefly a number of elementary identities and formulas for some of the special functions needed. In particular, the function defined in (\ref{eq:13}) and the second solution of the deterministic Lam\'e equation relies on
the Jacobi Epsilon function, which is detailed below. Of course a very comprehensive collection of results for Jacobian Elliptic Functions is contained in the website \cite{SF}.
\begin{align*}
\text{sn}^{2}(x,q) + \text{cn}^{2}(x,q) &=1\\
\text{dn}^{2}(x,q) + q^{2}\text{sn}^{2}(x,q)& =1\\
\text{dn}^{2}(x,q) -q^{2} \text{cn}^{2}(x,q) &=1 -q^{2} =q'^{2}\\
\text{am}(x,q) &= \int_{0}^{x}\text{dn}(t,q)dt\\
\text{sn}(x,q)& = \sin(\text{am}(x,q))\\
E(\phi,q) &= \int_{0}^{\phi}\sqrt{1 - q^{2}\sin^{2}\theta}d\theta.
\end{align*}
Jacobi's Epsilon function:
\begin{equation*}
 \mathcal{E}(x,q)= \int_{0}^{x}\text{dn}^{2}(t,q)dt\left(=E(\text{am}(x,q),q) ~  \text{ for }-K\leq x \leq K \right).
\end{equation*}
Asymptotics of Jacobi's Epsilon function and relation to Theta functions:
\begin{equation}\label{eq:21}
\mathcal{E}(x,q) = \frac{E(q)}{K(q)}x + \frac{\frac{d\theta_{4}}{d\xi}((\xi,p))}{\theta^{2}_{3}(0,p)\theta_{4}(\xi,p)}~,
\end{equation}
with $ \xi = x/\theta^{2}_{3}(0,p) $ and $\displaystyle p =\exp\{-\pi K'(q)/K(q)\} $, $ K'(q) = K(q') $, $ q'^{2} + q^{2} =1
$. The logarithmic derivative in (\ref{eq:21}) can be expressed as:
\begin{equation*}
\frac{\frac{d\theta_4}{dz}(z,p)}{\theta_4(z,p)} = 4\sum_{n=1}^{\infty}\frac{p^{n}}{1
  - p^{2n}}\sin(2nz)
\end{equation*}

\end{document}